\documentclass{article}
\usepackage{amssymb}
\usepackage{amsmath}
\usepackage{amsthm}
\usepackage{subcaption}
\usepackage[noabbrev]{cleveref}

\newcommand{\n}{\mathfrak{n}}

\newcommand{\h}{\mathfrak{h}}
\newtheorem{dfn}{Definition}[section]
\newtheorem{rmk}{Remark}[section]
\newtheorem{thm}{Theorem}[section]
\newtheorem{prop}{Proposition}[section]

\newtheorem{lem}{Lemma}[section]
\newtheorem{ex}{Example}[section]

\newtheorem{sub}{Sublemma}[section]
\newtheorem{fact}{Fact}[section]
\newtheorem{ques}{Question}[section]

\usepackage{amscd}
\newcommand{\R}{\mathbb{R}}

\newcommand{\Z}{\mathbb{Z}}
\newcommand{\N}{\mathbb{N}}
\usepackage[dvipdfmx]{graphicx}
\begin{document}

\title{On the speed of convergence to the asymptotic cone for non-singular nilpotent groups}
\author{Kenshiro Tashiro}
\date{}
\maketitle
\begin{abstract}

	We study the speed of convergence to the asymptotic cone for Cayley graphs of nilpotent groups.
	Burago showed that $\{(\Z^d,\frac{1}{n}\rho,id)\}_{n\in\N}$ converges to $(\R^d,d_{\infty},id)$ and its speed is $O(\frac{1}{n})$ in the sense of Gromov-Hausdorff distance.
	Later Breuillard and Le Donne gave estimates for non-abelian cases,
	and constructed an example whose speed of convergence is slower than $O(\frac{1}{\sqrt{n}})$.

	For $2$-step nilpotent groups,
	we show that
	if the Mal'cev completion is non-singular,
	then the speed of convergence is $O(\frac{1}{n})$ for any choice of generating set.
In terms of subFinsler geometry,
this condition is also equivalent to the strongly blacket generating condition,
and to the absence of abnormal geodesics on the asymptotic cone.
\end{abstract}

\section{Introduction}\label{sec0}

Let $\Gamma$ be a torsion free nilpotent group generated by a finite symmetric subset $S\subset \Gamma$,
and $\rho_S$ the associated word metric.
The asymptotic cone of $(\Gamma,\rho_S,id)$ is the Gromov-Hausdorff limit of the sequence $\{(\Gamma,\frac{1}{n}\rho_S,id)\}_{n\in\N}$.
In general,
the existence and the uniqueness of the limit is not trivial,
however,
Pansu showed that the asymptotic cone of $(\Gamma,\rho_S,id)$ is uniquely determined up to isometry in \cite{pan}.
The limit space $(N,d_{\infty},id)$ is a simply connected nilpotent Lie group endowed with a subFinsler metric (see section \ref{sec23}).
In particular,
if $\Gamma$ is a $2$-step nilpotent group,
$N$ is isomorphic to the Mal'cev completion of $\Gamma$.

The asymptotic cone and the original metric space are sometimes close in the following sense.
Burago \cite{bur} showed that a Cayley graph of every free abelian group is $(1,C)$-quasi-isometric to its asymptotic cone for some $C>0$.
This implies that the unit ball of a scaled down Cayley graph centered at the identity,
denote $B_{\frac{1}{n}\rho_S}(1)$,
converges to that of the asymptotic cone rapidly.
Namely,
$$d_{GH}(B_{\frac{1}{n}\rho_S}(1),B_{d_{\infty}}(1))=O(n^{-1}),$$
where $d_{GH}$ is the Gromov-hausdorff distance.

Motivated by this result,
Gromov \cite{gro2} asked whether a Cayley graph of a nilpotent group is $(1,C)$-quasi-isometric to its asymptotic cone,
and if not,
what is the speed of convergence.
The first result on non-abelian nilpotent groups is given by Krat \cite{kra},
who showed that the discrete $3$-Heisenberg group $H_3(\Z)$ endowed with a word metric is $(1,C)$-quasi isometric to its asymptotic cone.
For general cases,
Breuillard and Le Donne first gave estimates in \cite{bre2}.
Later the result is sharpened by Gianella \cite{gia},
who showed that
$$d_{GH}(B_{\frac{1}{n}\rho_S}(1),B_{d_{\infty}}(1))=O(n^{-\frac{1}{r}}),$$
where $r$ is the nilpotency class of $\Gamma$.
Moreover,
Breuillard and Le Donne also showed in \cite{bre2} that for the $2$-step nilpotent group $\Z\times H_3(\Z)$,
there is a generating set such that the estimates $O(n^{-\frac{1}{2}})$ is sharp.
From this example,
Fujiwara \cite{fuj2} asked the following question.

\begin{ques}[Question 4 in \cite{fuj2}]\label{ques1}
Let $\Gamma$ be a lattice in a simply connected strictly non-singular nilpotent Lie group,
and $\rho$ a $\Gamma$-invariant proper coarsely geodesic pseudo metric.
Then are $(\Gamma,\rho)$ and its asymptotic cone $(1,C)$-quasi isometric for some $C>0$?
\end{ques}

Here we do not pursue the assumption on the metric such as coarsely geodesic condition,
but it is a large class of pseudo metrics which includes word metrics.
We will mention the group theoretic condition on $\Gamma$.
A simply connected nilpotent Lie group $N$ is {\it strictly non-singular} if for all $z\in Z(N)$,
the center of $N$,
and all $x\in N\setminus Z(N)$,
there is $y\in N$ such that $[x,y]=x^{-1}y^{-1}xy=z$.
For $2$-step nilpotent groups,
that condition is simply called {\it non-singular},
defined as below.
\begin{dfn}
	A simply connected $2$-step nilpotent Lie group $N$ is called non-singular if for all $z\in [N,N]$ and all $x\in N\setminus [N,N]$,
	there is $y\in N$ such that $[x,y]=z$.
\end{dfn}

We answer Question \ref{ques1} in the case where $\Gamma$ is non-singular,
in otherwords $2$-step,
and the endowed metric $\rho$ is a word metric.

\begin{thm}\label{thm1}
	Let $\Gamma$ be a lattice of a simply connected non-singular $2$-step nilpotent group $N$,
	and $\rho_S$ a word metric on $\Gamma$.
	Then there is $C>0$ such that $(\Gamma,\rho_S)$ is $(1,C)$-quasi isometric to its asymptotic cone.
\end{thm}

Via the exponential map from the associated Lie algebra $\n$ to $N$,
the non-singular condition is equivalent to every bracket generating subspaces in the Lie algebra being \textit{strongly bracket generating}.
Here a subspace $V\subset \n$ is called strongly bracket generating if for any $X\in V\setminus\left\{0\right\}$,
$\n=V\oplus[X,V]$.
By Theorem A.1 in \cite{led},
strongly bracket generating condition,
equivalently non-singular condition,
is equivalent to the abcense of abnormal curves.

\begin{rmk}
	After the first draft of this paper is completed, we are informed
	by Emmanuel Breuillard that Theorem \ref{thm1} and an argument 
in the same line as our proof are 
known to specialists including him,
but it does not exist
in the literature yet,
and we feel it is worth publishing it. 
\end{rmk}
Theorem \ref{thm1} is on a finitely generated group,
which is related to a claim on a nilpotent Lie group by using the following result by Stoll.

\begin{prop}[Proposition 4.3 in \cite{sto}]\label{prop0}
	Let $\Gamma$ be a finitely generated torsion free $2$-step nilpotent group,
$\rho_S$ a word metric on $\Gamma$,
	and $N$ the Mal'cev completion of $\Gamma$.
	Then there is a left invariant subFinler metric $d_S$ on $N$ and $C>0$ such that $(\Gamma,\rho_S)$ is $(1,C)$-quasi isometric to $(N,d_S)$ by the natural inclusion map.

\end{prop}

He constructed such a metric $d_S$ explicitly,
now called the {\it Stoll metric}.

It is easy to see that the asymptotic cones of $(\Gamma,\rho_S,id)$ and $(N,d_S,id)$ are isometric,
hence the following theorem implies Theorem \ref{thm1}.

\begin{thm}[Precisely in Theorem \ref{thmneo}]\label{thm3}
	Let $N$ be a simply connected non-singular $2$-step nilpotent Lie group endowed with a left invariant subFinsler metric $d$.
	Then there is $C>0$ such that $(N,d)$ is $(1,C)$-quasi isometric to its asymptotic cone.
\end{thm}

\begin{rmk}
	In the original setting of Question \ref{ques1},
	that is the metric $\rho$ is a coarsely geodesic metric,
our method cannot be applied because of the following reason.
Roughly speaking,
we show the main result by constructing a path in $(N,d_{\infty})$ from a geodesic in $(N,d)$ and vice versa.
The scheme is;
\begin{enumerate}
	\item Project the geodesic $c_0$ in $(N,d)$ onto its abelianized normed space,
		say $c_1$.
	\item Construct a path $c_2$ in the abelianization of $(N,d_{\infty})$ which is close to the $c_1$ in the (Gromov--)Hausdorff sense.
	\item Lift up the $c_2$ to a path $c_3$ in $(N,d_{\infty})$.
	\item Slight variation of the $c_3$ can have the same endpoints with the $c$.
\end{enumerate}
Finally we find that the length of $c_3$ is same to the $c_0$ up to constant.

In coarsely geodesic setting,
the second step is impossible since geodesics on $\R^n$ with a $\Z^n$-invariant metric may be quite far from the straight segment in the Hausdorff sense (cf. \cite{bur2} and \cite{ban}).

\end{rmk}

\section*{Acknowledgement}

The author would like to express his great thanks to Professor Koji Fujiwara for many helpful suggestions and comments.
He would like to offer his appreciation to Professor Emmanuel Breuillard for informing him of precise research status.
He would like to be grateful to Professor Enrico Le Donne for sharing information on his work with him and bringing \cite{led} into his attention.

\section{The asymptotic cone of a nilpotent Lie group endowed with a left invariant subFinsler metric}\label{sec2}

Let $N$ be a simply connected $2$-step nilpotent Lie group,
and $d$ a left invariant subFinsler metric on $N$.
In this section,
we shall construct the asymptotic cone of $(N,d,id)$.

\subsection{Nilpotent Lie groups and nilpotent Lie algebras}\label{sec21}

Let $\n$ be the Lie algebra associated to $N$.
It is known that the exponential map from $\n$ to $N$ is a diffeomorphism.
By the Baker-Campbell-Hausdorff formula,
the group operation on $N$ is written by
$$\exp(X)\cdot\exp(Y)=\exp(X+Y+\frac{1}{2}[X,Y]).$$
In particular,
we can identify the commutator on $N$ and the Lie bracket on $\n$ as
$$[\exp(X),\exp(Y)]=\exp([X,Y]).$$
Hence we sometimes identify elements in $N$ and $\n$ via the exponential map.

Let $V_{\infty}$ be a subspace of $\n$ such that
$$V_{\infty}\cap[\n,\n]=(0) ~~\text{and}~~ V_{\infty}+[\n,\n]=\n.$$
Then $\n$ is spanned by the direct sum $V_{\infty}\oplus[\n,\n]$ and any element in $\n$ will be written by $X+Y$,
where $X\in V_{\infty}$ and $Y\in[\n,\n]$.

To such a decomposition,
we can define the following two endomorphisms of $N$ and $\n$.

\noindent\underline{$\delta_t$:\ Dilation}\\
We may associate a Lie algebra automorphism $\delta_t:\n\to\n$ $(t\in\R_{>0})$ which is determined by
$$\delta_t(X+Y)=tX+t^2Y.$$
This Lie algebra automorphism is called the {\it dilation}.
It induces the diffeomorphism of $N$ via the exponential map
(we also denote that diffeomorphism by $\delta_t$).

\noindent\underline{$\pi$:\ Projection to $V_{\infty}$}\\
	Set a mapping $\pi:\n\to V_{\infty}$ by $\pi(X+Y)=X$.
By the Baker-Campbell-Hausdorff formula,
it is easy to see that $\pi\circ\log:N\to V_{\infty}$ is a surjective group homomorphism,
as we see $V_{\infty}$ an abelian Lie group.
We will simply denote the homomorphism $\pi\circ\log$ by $\pi$.

\subsection{Left invariant subFinsler metrics}\label{sec22}

Let $N$ be a connected Lie group with the associated Lie algebra $\n$.
Suppose a vector subspace $V\subset\n$ and a norm $\|\cdot\|$ on $V$ are given.
Then $V$ induces the left invariant subbundle $\Delta$ of the tangent bundle of $N$.
Namely,
a vector $v$ at a point $p\in N$ is an element of $\Delta$ if $(L_p)^{\ast}v\in V$.
For such $v$,
we set $\|v\|:=\|(L_p)^{\ast}v\|$.
This $\Delta$ is called a {\it horizontal distribution}.

One says that an absolutely continuous curve $c:[a,b]\to N$ with $a,b\in\R$ is {\it horizontal} if the derivative $\dot c(t)$ is in $\Delta$ for almost all $t\in[a,b]$.
Then for $x,y\in N$,
one may define a subFinsler metric as
$$d(x,y)=\inf\left\{\int_a^b\|\dot c(t)\|dt\Big|c~\text{is horizontal}~,c(a)=x,c(b)=y\right\}.$$
Note that such $d$ is left invariant.

Chow showed that any two points in $N$ are connected by a horizontal path if and only if $V$ is {\it bracket generating},
that is,
$$V+[V,V]+\cdots+\underbrace{[V,[V,[\cdots]\cdots]}_r=\n.$$
In particular,
the subspace $V_{\infty}$,
given in Section \ref{sec21},
is bracket generating.

\subsection{The asymptotic cone}\label{sec23}

Roughly speaking,
an asymptotic cone is a metric space which describes how a metric space looks like when it is seen from very far.
This is characterized by the Gromov-Hausdorff distance between compact metric spaces.

\begin{dfn}
Let $(X,d,p)$ be a pointed proper metric space.
If the sequence of pointed proper metric spaces $\{(X,\frac{1}{n}d,p)\}_{n\in\N}$ converges to a metric space $(X_{\infty},d_{\infty},p_{\infty})$ in the Gromov-Hausdorff topology,
then $(X_{\infty},d_{\infty},p_{\infty})$ is called the asymptotic cone of $(X,d,p)$.
\end{dfn}

\begin{rmk}
It is not trivial whether the limit exists or not.
For nilpotent Lie groups endowed with left invariant subFinsler metrics,
the existence and the uniqueness of the limit is shown in \cite{bre}.
For more precise information,
see \cite{van}
\end{rmk}

Let us recall the definition of the Gromov-Hausdorff topology on the set of pointed proper metric spaces.
A sequence of pointed proper metric spaces $\{(X_n,d_n,p_n)\}_{n\in\N}$ is said to converge to the pointed metric space $(X_{\infty},d_{\infty},p_{\infty})$ if for any $R>0$,
the sequence of metric balls $\{B_{d_n}(p_n,R)\}_{n\in\N}$ converges to $B_{d_{\infty}}(p_{\infty},R)$ in the Gromov-Hausdorff topology on the set of compact metric spaces.

The Gromov-Hausdorff topology on the set of compact metric spaces is characterized by the Gromov-Hausdorff distance.
For compact metric spaces $(X,d_X)$ and $(Y,d_Y)$,
it is determined by
$$d_{GH}(X,Y):=\inf\left\{d_{H,Z}(X,Y)\Big|Z=X\sqcup Y,d_Z|_X=d_X,d_Z|_Y=d_Y\right\},$$
Here $d_{H,Z}$ is the Hausdorff distance on compact subsets on $Z$,
namely the smallest $r>0$ such that $X$ lies in the $r$-neighborhood of $Y$ and $Y$ lies in the $r$-neighborhood of $X$.

\vspace{12pt}

Suppose a left invariant subFinsler metric $d$ on $N$ is determined by a bracket generating subspace $V\subset \n$ and a norm $\|\cdot\|$ on $V$.
By using the homomorphism $\pi$,
define a left invariant subFinsler metric $d_{\infty}$ on $N$ which is determined by the subspace $V_{\infty}\subset\n$ and the norm $\|\cdot\|_{\infty}$ on $V_{\infty}$ whose unit ball is $\pi(B_{\|\cdot\|}(1))$,
where $B_{\|\cdot\|}(1)$ is the unit ball of the normed space $(V,\|\cdot\|)$ cntered at $0$.

\begin{thm}[Theorem 3.2 in \cite{bre2}]\label{thm210}

For any sequence $\{g_i\}_{i\in\N}$ on $N$ such that $d(g_i)\to\infty$ as $i\to\infty$,
$$\lim_{i\to\infty}\frac{d_{\infty}(g_i)}{d(g_i)}=1.$$

In particular,
the asymptotic cone of $(N,d,id)$ is isometric to $(N,d_{\infty},id)$.
\end{thm}

The pair $(N,V_{\infty})$ is called a {\it Carnot group}.
If a subFinsler metric is induced from a Carnot group,
such as $d_{\infty}$,
then it satisfies the following properties.

\begin{fact}\label{fact20}
\begin{itemize}
\item[(a)] For every horizontal path $c$,
$$length(c)=length(\pi\circ c).$$
In particular,
$$\|\pi(g)\|_{\infty}\leq d_{\infty}(g),$$
and the equality holds if and only if $g\in \exp(V_{\infty})$.
\item[(b)] For $x,y\in N$,
$$d_{\infty}(\delta_t(x),\delta_t(y))=td_{\infty}(x,y).$$
\end{itemize}
\end{fact}

Notice that a general subFinsler metric,
such as $d$,
does not satisfies Fact \ref{fact20}.

\begin{rmk}
	\begin{itemize}
		\item By its definition,
			$\pi|_V$ sends $R$-balls in
$(V,\|\cdot\|)$ onto $R$-balls in $(V_{\infty},\|\cdot\|_{\infty})$.

\item By Fact \ref{fact20}(a),
$\pi$ sends $R$-balls in $(N,d_{\infty})$ onto $R$-balls in $(V_{\infty},\|\cdot\|_{\infty})$.
\item In Lemma \ref{prop1},
	we shall see that $\pi$ sends $R$-balls in $(N,d)$ onto $R$-balls in $(V_{\infty},\|\cdot\|_{\infty})$.
\end{itemize}
\end{rmk}

\section{Geodesics in $(N,d)$}\label{sec9}

Let $N$ be a simply connected $2$-step nilpotent Lie group,
and $d$ a left invariant subFinsler metric on $N$ determined by a subspace $V\subset \n$ and a norm $\|\cdot\|$ on $V$.
In this section,
we study geodesics in $(N,d)$.

For $g\in (N,d)$,
let $c$ be a geodesic from $id$ to $g$ with its length $t=d(g):=d(id,g)$.
Divide $c$ into $M$ pieces so that each lengths are $\frac{t}{M}$.
In other words,
$c$ is the concatenation of paths $c_i:[0,\frac{t}{M}]\to N$,
$i=1,\dots ,M$,
which are geodesics from $id$ to $h_i=c(\frac{i-1}{M}t)^{-1}c(\frac{i}{M}t)$.
Notice that $g=h_1\cdots h_{M}$.
Set
$$I(c,M,R)=\left\{i\in\{1,\dots M\}~|~\|\pi(h_i)\|_{\infty}< Rd(h_i)=R\frac{t}{M}\right\}$$
for $0<R\leq1$.
The goal of this section is to show the following proposition.

\begin{prop}\label{prop90}

There exists $K>0$ such that for any $M\in\N$,
if $t\geq M$ then
$$|I(c,M,R)|\leq \frac{K}{(1-R)^2}.$$

\end{prop}

\begin{ex}[The $3$-Heisenberg Lie group with a subFinsler metric]\label{ex0}
	The $3$-Heisenberg Lie group $H_3(\R)$ is the $2$-step nilpotent Lie group diffeomorphic to $\R^3$ equipped with a group operation
$$(x_1,y_1,z_1)\cdot(x_2,y_2,z_2)=(x_1+x_2,y_1+y_2,z_1+z_2+\frac{x_1y_2-x_2y_1}{2}).$$
The associated Lie algebra $\h_3$ is spanned by three vectors $\{X,Y,Z\}$ such that $[X,Y]=Z$,
and its derived Lie algebra is $[\h_3,\h_3]=Span(Z)$.
Then $V_{\infty}=\langle X,Y\rangle\subset \h_3$ and we can identify it to the plane $\{(x,y,0)\}\subset H_3$ via the exponential map.

\vspace{10pt}
(1)Let $\|\cdot\|_1$ be the $l^1$ norm on a vector subspace $V_{\infty}$,
and $d_1$ the induced left invariant subFinsler metric on $(H_3,V_{\infty},\|\cdot\|_1)$.

The shape of geodesics in $(H_3,d_1)$ is given in \cite{duc}.
For example,
a geodesic $c$ from $(0,0,0)$ to $(0,0,\frac{t^2}{16})$ is the concatenation of $4$ linear paths as in Figure \ref{fig1}.
Here we say a curve is linear if it is represented by $c(t)=\exp(tX)$ for $X\in V_{\infty}$.
We can catch precise shape of geodesics by projecting the curve to the plane $\{(x,y,0)\}$.
As in Figure \ref{fig2},
it starts and ends at $(0,0)$ forming the square.

Divide $c$ into $4$-pieces and denote them by $c_i$ $(i=1,2,3,4)$.
Then $c_i$'s are the linear paths.
It is easy to see that $length(c_i)=length(\pi\circ c_i)=\frac{t}{4}$ for all $i$.
Hence $I(c,4,1)=0$ independent of $t$.

\begin{figure}[h]
\begin{subfigure}{0.5\textwidth}
    \def\svgwidth{0.9\columnwidth}
    \vspace{-2pt}\includegraphics[width=6cm]{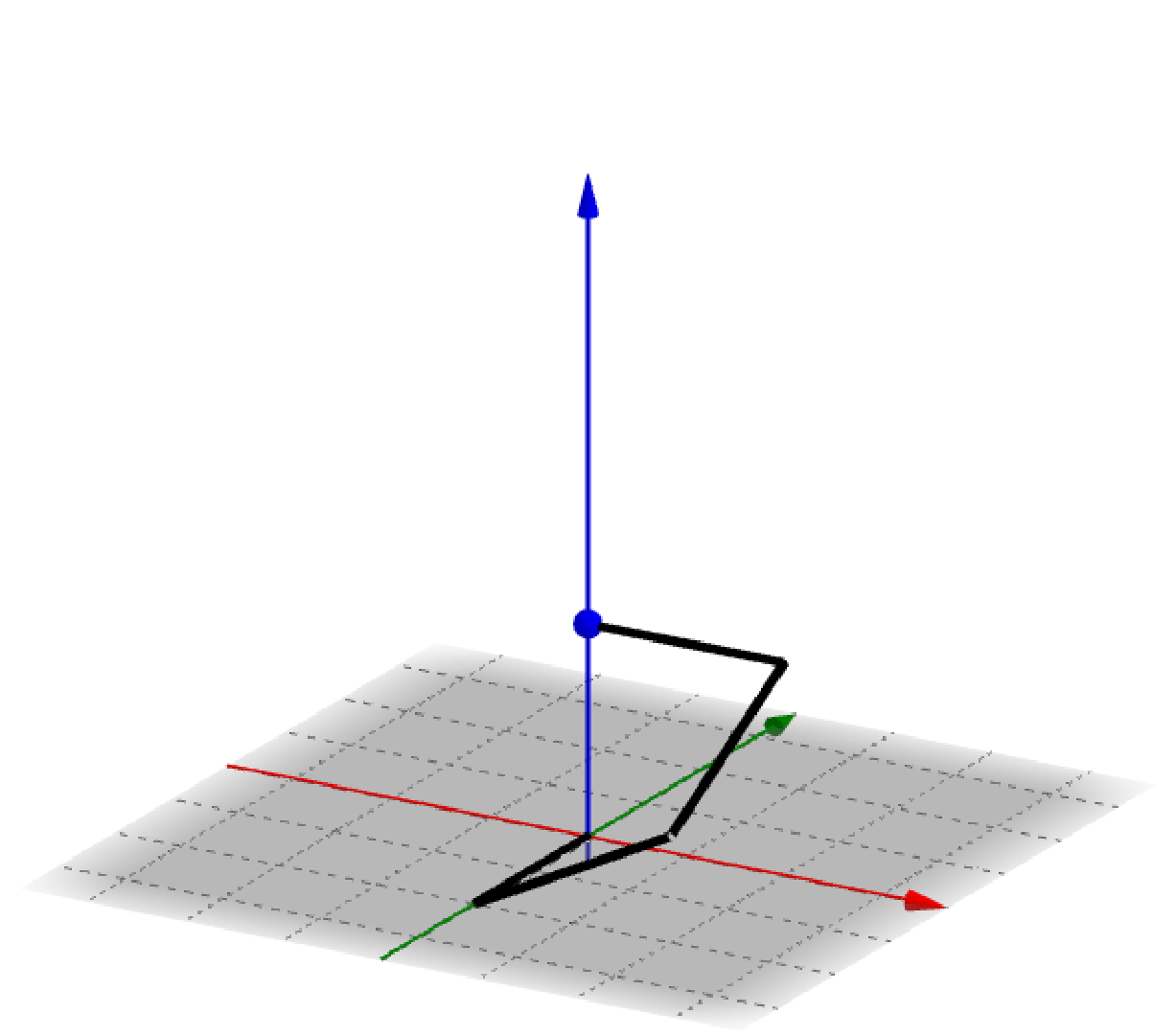}
    \caption{$(H_3,d_1)$}
    \label{fig1}
\end{subfigure}
\begin{subfigure}{0.5\textwidth}
    \def\svgwidth{0.9\columnwidth}
    \centering\hspace{1cm}\includegraphics[width=5.5cm]{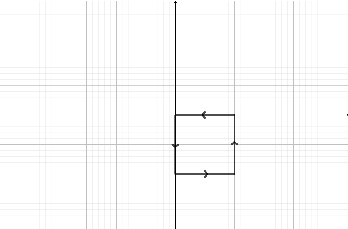}
    \caption{$(W,\|\cdot\|_1)$}
    \label{fig2}
\end{subfigure}
\caption{}
\end{figure}

\vspace{10pt}
(2)Let $\|\cdot\|_2$ be the $l^2$ norm on $V_{\infty}$,
and $d_2$ the induced subFinsler (subRiemannian) metric on $(H_3,V_{\infty},\|\cdot\|_2)$.
A geodesic $c$ from $(0,0,0)$ to $(0,0,\frac{t^2}{4\pi})$ is given as in Figure \ref{fig3}.
If the geodesic is projected to $\left\{(x,y,0)\right\}$ by $\pi$,
then the projected path starts and ends at $(0,0)$ rounding the circle of radius $\frac{t}{2\pi}$.
This curve is not a concatenation of linear paths,
however Proposition \ref{prop90} holds.

Notice that the length of $c$ is $t$,
which is the circumference of the projected circle in $V_{\infty}$.
As in Figure \ref{fig4},
divide $c$ into $4$ pieces,
and denote them by $c_i$ $(i=1,2,3,4)$.
Each arc $c_i$'s have length $\frac{t}{4}$.
On the other hand,
each chords in Figure \ref{fig4} is a geodesic in $(W,\|\cdot\|_2)$ whose length is $2\frac{t}{2\pi}\sin(\frac{\pi}{4})=\frac{t}{\pi\sqrt{2}}$.
Hence $h_i$'s,
the endpoints of $c_i$'s,
satisfy
$$\|\pi(h_i)\|_2=\frac{t}{\pi\sqrt{2}}.$$
It means that $I(c,4,R)=0$ for $R\leq \frac{2\sqrt{2}}{\pi}$.

\begin{figure}[h]
\begin{subfigure}{0.5\textwidth}
    \def\svgwidth{0.9\columnwidth}
    \vspace{4pt}\includegraphics[width=7cm]{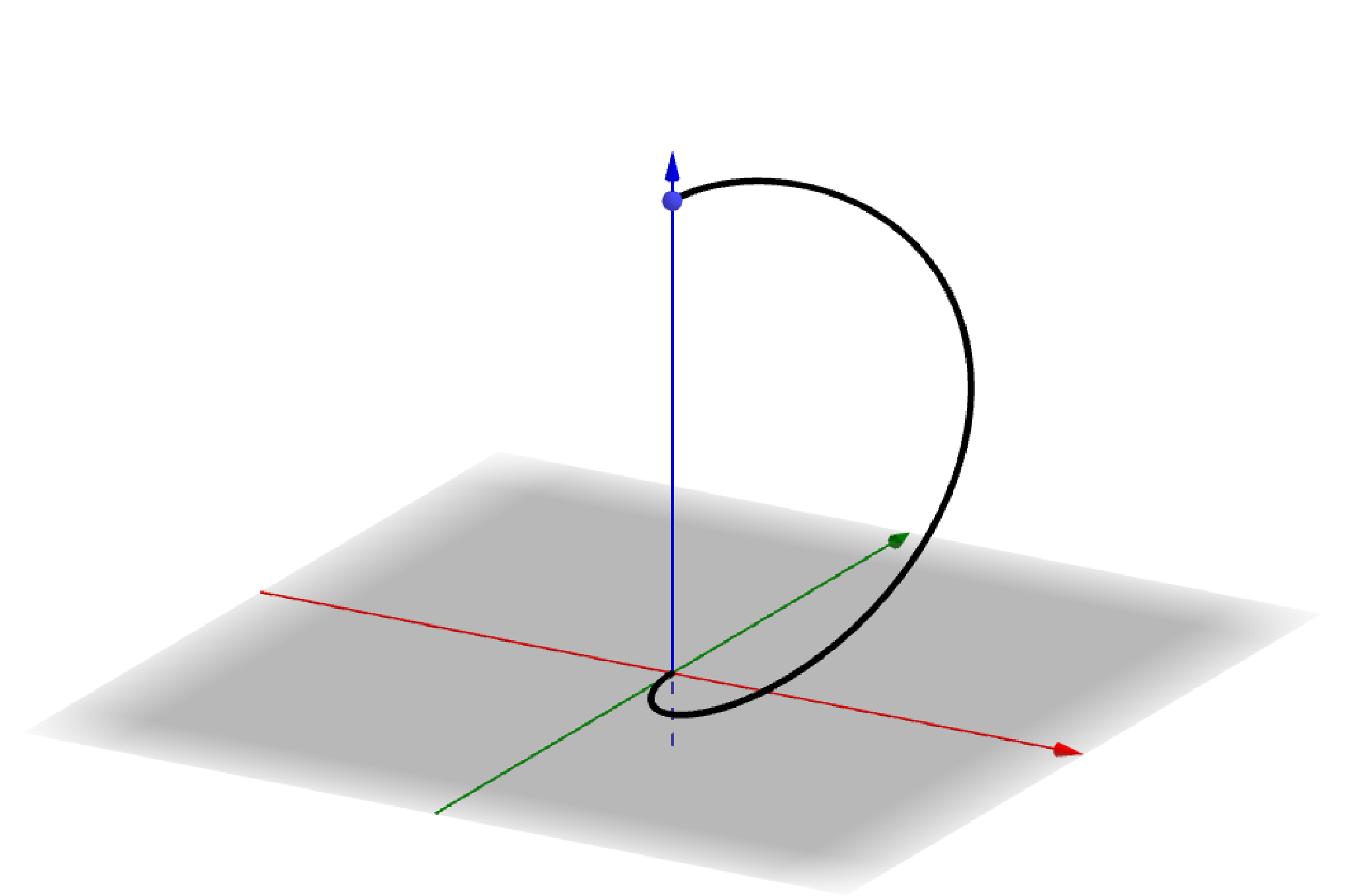}
    \caption{$(H_3,d_2)$}
    \label{fig3}
\end{subfigure}
\begin{subfigure}{0.5\textwidth}
    \def\svgwidth{0.9\columnwidth}
    \centering\hspace{1cm}\includegraphics[width=5cm]{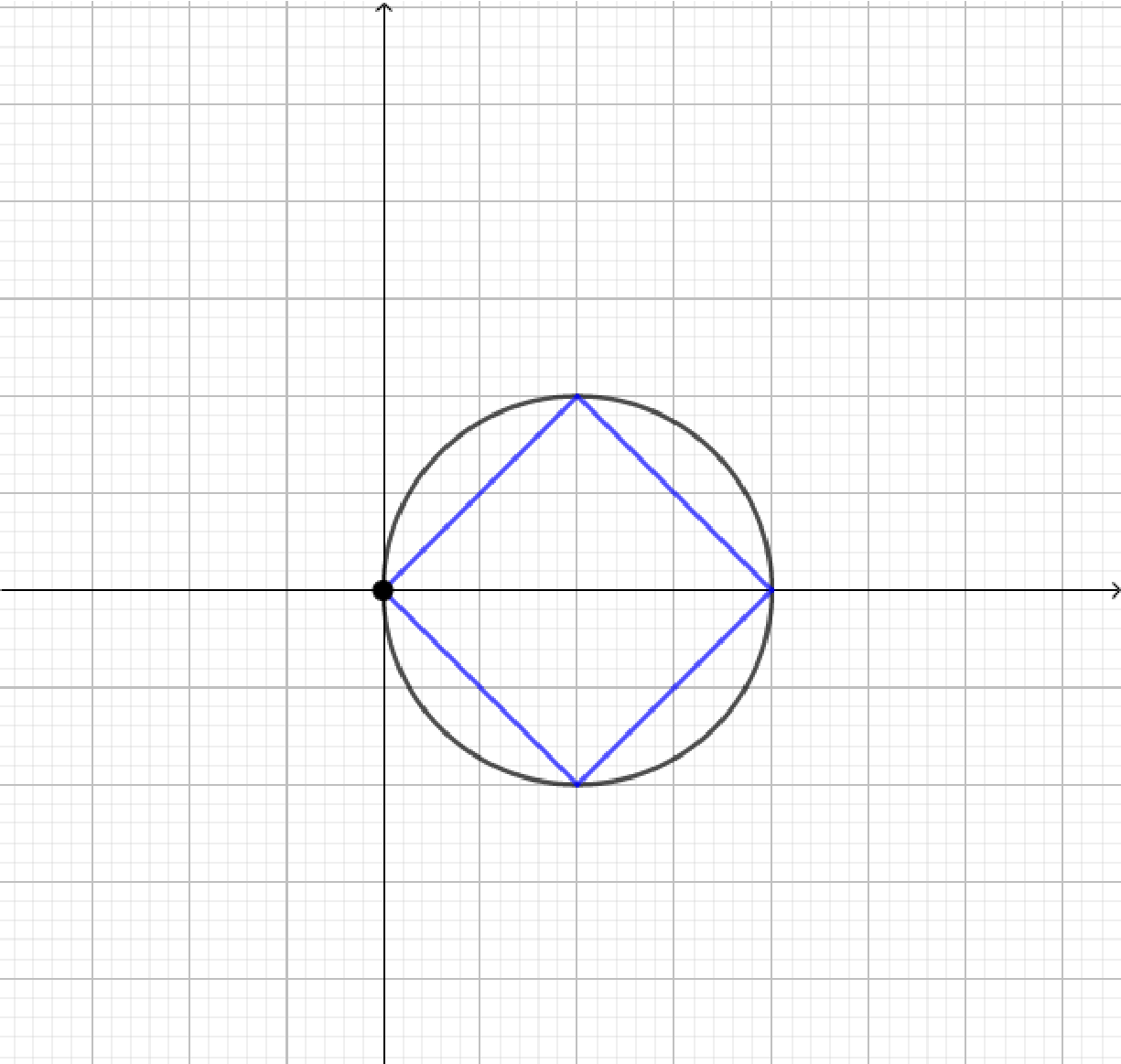}
    \caption{$(W,\|\cdot\|_2)$}
    \label{fig4}
\end{subfigure}
\caption{}
\end{figure}

\end{ex}

\vspace{20pt}
We start to prove easy lemmas.
Fix a norm $\|\cdot\|_{[N,N]}$ on $[N,N]$.

\begin{lem}\label{lem90}
There exists $K_1>0$ such that for any $r\geq1$,
$$\sup\{\|g^{-1}h\|_{[N,N]}~|~g,h\in B_{d_{\infty}}(r),g^{-1}h\in[N,N]\}=K_1r^2.$$
\end{lem}

\begin{proof}

Take $g,h\in B_{d_{\infty}}(r)$ so that $g^{-1}h\in[N,N]$.
By Fact \ref{fact20}(b),
$g'=\delta_{\frac{1}{r}}(g)$ and $h'=\delta_{\frac{1}{r}}(h)$ are in $B_{d_{\infty}}(1)$.

Set $X=X_1+X_2=\log(g)$ and $Y=X_1+Y_2=\log(h)$,
where $X_1\in V_{\infty}$ and $X_2,
Y_2 \in[\n,\n]$.
Here we can take the common $X_1$ since $g^{-1}h\in[N,N]$.
By the definition of $\delta_{\frac{1}{r}}$,
$\log(g')=\frac{1}{r}X_1+\frac{1}{r^2}X_2$ and
$\log(h')=\frac{1}{r}X_1+\frac{1}{r^2}Y_2.$

Then
\begin{align*}
g^{-1}h&=\exp(-X_1-X_2)\exp(X_1+Y_2)\\
&=\exp(-X_2+Y_2)\\
&=\exp(r^2\frac{1}{r^2}(-X_2+Y_2))\\
&=\exp(\frac{1}{r^2}(-X_2+Y_2))^{r^2}\\
&=(g^{\prime -1}h')^{r^2}.
\end{align*}
We obtain the desired equality
\begin{align*}
&\sup\left\{\|g^{-1}h\|_{[N,N]}~|g,h\in B_{d_{\infty}}(r),g^{-1}h\in[N,N]\right\}\\
&=r^2\sup\{\|x^{-1}y\|_{[N,N]}~|x,y\in B_{d_{\infty}}(1),x^{-1}y\in[N,N]\}\\
&=K_1r^2,
\end{align*}
where $K_1=\sup\{\|x^{-1}y\|_{[N,N]}~|x,y\in B_{d_{\infty}}(1),x^{-1}y\in[N,N]\}<\infty$.
\end{proof}

\begin{lem}\label{lem89}
	For any $h\in[N.N]\setminus\{id\}$,
	there are $X,Y\in\partial B_{\|\cdot\|_{\infty}}(1)$ such that
	\begin{equation}\label{eq7}
		[\exp(X),\exp(Y)]\in h^{\R_{>0}}=\exp(\R_{>0}\log(h)).
	\end{equation}
	Moreover,
assume that $\|[\exp(X),\exp(Y)]\|_{[N,N]}$ is maximal within the condition (\ref{eq7}).
	Then there is $L=L(d_{\infty},\|\cdot\|_{[N,N]})>0$,
	independent of $h$,
	such that  $\|[\exp(X),\exp(Y)]\|_{[N,N]}\geq L$
\end{lem}

\begin{proof}
	Since $h\in[N,N]\setminus\left\{id\right\}$,
there are $x,y\in N$ such that $h=[x,y]$.
Set $X_1,Y_1\in V_{\infty}$ and $X_2,Y_2\in[\n,\n]$ such that
$X_1+X_2=\log(x)$ and $Y_1+Y_2=\log(y)$.
Since $X_2,Y_2$ are in the center of $\n$,
$$[\pi(x),\pi(y)]=[X_1,Y_1]=[X_1+X_2,Y_1+Y_2]=\log(h).$$
Set $X'=\frac{1}{\|X_1\|_{\infty}}X_1$ and $Y'=\frac{1}{\|Y_1\|_{\infty}}Y_1$,
It is clear that these $X',Y'$ are in $\partial B_{\|\cdot\|_{\infty}}(1)$ and satisfy $[\exp(X'),\exp(Y')]\in h^{\R_{>0}}$.
It completes the former part of this lemma.

The latter part is trivial since the restricton of the commutator $[\cdot,\cdot]$ to $\exp(V_{\infty})\times \exp(V_{\infty})\subset N\times N$ is a submersion.
\end{proof}

Next we study a length preserving translation of a element in $(V_{\infty},\|\cdot\|_{\infty})$ to $(V,\|\cdot\|)$ and vice versa.
\begin{lem}\label{prop1}

For any $g\in N$,
there exists $Y_g\in\pi|_V^{-1}(\pi(g))$ such that

\begin{itemize}
	\item $\|Y_g\|=\|\pi(g)\|_{\infty}=d(\exp(Y_g))=\inf\{d(h)|h\in\pi^{-1}(\pi(g))\},$
\item An infinite path $c:\R_{\geq 0}\to N$,
$t\mapsto\exp\left(t\frac{Y_g}{\|Y_g\|}\right)$ is a geodesic ray i.e. for any $t_1,t_2\in\R_{\geq 0}$,
$d(c(t_1),c(t_2))=|c_1-c_2|$.
\end{itemize}
\end{lem}

\begin{proof}

	From the construction of the asymptotic cone of $(N,d,id)$,\\
	$\pi|_V(B_{\|\cdot\|}(R))=B_{\|\cdot\|_{\infty}}(R)$ for any $R>0$.
Thus for any $g\in N$,
we can take $Y_g$ in $V$ such that $\|Y_g\|=\|\pi(g)\|_{\infty}$.

\vspace{12pt}
We shall see that this $Y_g$ is the desired one.
Clearly $\|Y_g\|\geq d\left(\exp(Y_g)\right)$ since the curve $c:[0,\|Y_g\|]\to N$,
$c(t)=\exp\left(t\frac{Y_g}{\|Y_g\|}\right)$ is a horizontal path from $id$ to $\exp(Y_g)$ such that $length(c)=\|Y_g\|$.

We claim the converse by showing the inequality
\begin{equation}\label{eq32}
	\|\pi(g)\|_{\infty}\leq d\left(\exp(Y_g)\right).
\end{equation}
Let $c_1:\left[0,d\left(\exp(Y_g)\right)\right]\to N$ be a geodesic from $id$ to $\exp(Y_g)$ in $(N,d)$.
Then we obtain the horizontal path $c_2$ in $(N,d_{\infty})$ by letting the derivative $c_2^{\prime}(t)=\pi(c_1^{\prime}(t))$ for each $t\in[0,d\left(\exp(Y_g)\right)]$.
Since $\pi$ is distance non-increasing,
$length(c_2)\leq length(c_1)$.
By using Fact \ref{fact20}(a),
$\pi\circ c_2$ is a path in $V_{\infty}$ from $id$ to $\pi(Y_g)=\pi(g)$ whose length equals that of $c_2$.
Now we have constructed the path $\pi\circ c_2$ in $(V_{\infty},\|\cdot\|_{\infty})$ from $id$ to $\pi(g)$ whose length is shorter than $length(c_1)$,
which yields the inequality (\ref{eq32}).

The construction of $\pi\circ c_2$ from $c_1$ is applied to any $h\in \pi^{-1}(\pi(g))$ and any geodesic $c_1$ from $id$ to $h$.
Hence the inequality $d(h)\geq \|\pi(g)\|_{\infty}$ holds.
This argument yields the last part of the equality.

\vspace{12pt}
The second part of this lemma follows in the same way.
The above arguments imply that $c:[0,d(\exp(Y_g))]\to N$,
$c(t)=\exp\left(t\frac{Y_g}{\|Y_g\|}\right)$ is a geodesic from $id$ to $\exp(Y_g)$.
By the choice of $Y_g$,
we can show the second part of this lemma if $\|tY_g\|=\|t\pi(g)\|_{\infty}$ for $t\in\R_{\geq 0}$.
It is trivial since the mapping $\pi$ is a linear homomorphism.

\end{proof}

\begin{lem}[Proposition 2.13 in \cite{bre2}]\label{lem91}

There is $K_2>0$ such that for any $g\in N$,
$$\frac{1}{K_2}d(g)-K_2\leq d_{\infty}(g)\leq K_2d(g)+K_2.$$
\end{lem}

Now we pass to the proof of Proposition \ref{prop90}.
\begin{proof}[Proof of Proposition \ref{prop90}]
Fix $M\in\N$ and $0<R\leq 1$.
Let $c$ be a geodesic from $id$ to $g\in N$ with $length(c)=t\geq M$.
We consider an upper bound of the cardinality of $I=I(c,M,R)$.
Divide $c$ into $M$ pieces,
and denote each by $c_i$.
Let $h_i$ be the endpoint of $c_i$,
that is,
$h_i=c\left(\frac{t}{M}(i-1)\right)^{-1}c\left(\frac{t}{M}i\right)$.
Deform $c$ and $c_i$ as follows.
\begin{itemize}
	\item[(1)]If $i\in I$,
set $\tilde{c}_i:[0,\|Y_{h_i}\|]\to N$,
$$\tilde{c}_i(t)=\frac{Y_{h_i}}{\|Y_{h_i}\|}t,$$
where $Y_{h_i}$ are given as in Lemma \ref{prop1}

\item[(2)]If $i\notin I$,
set $\tilde{c}_i=c_i$.

\item[(3)]Set $\tilde{c}$ to be the concatenation of $\tilde{c}_i$'s starting at the identity.
\end{itemize}

\begin{figure}[h]
\begin{subfigure}{0.5\textwidth}
    \def\svgwidth{0.9\columnwidth}
    \includegraphics[width=6.5cm]{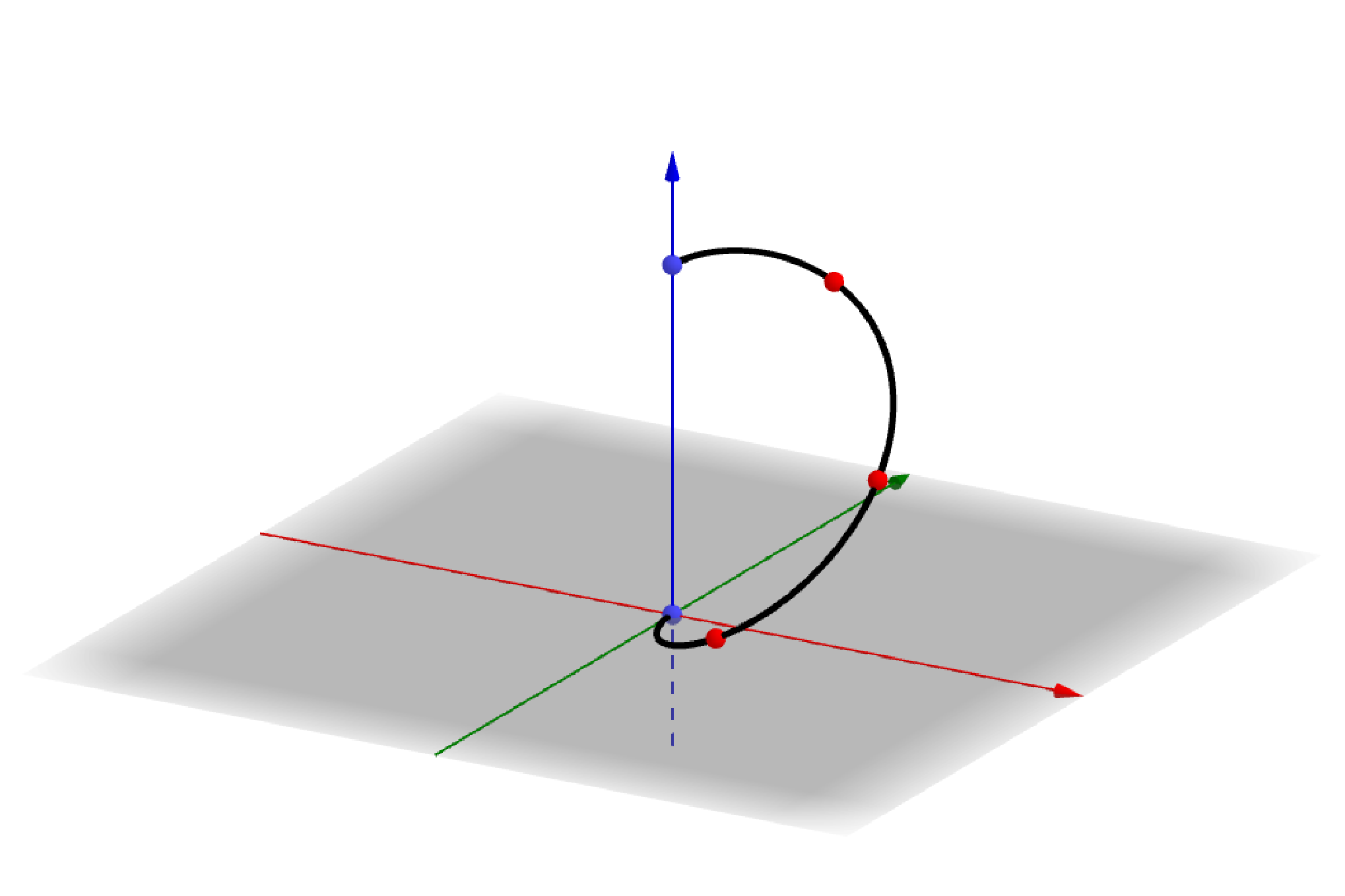}
    \caption{The path $c$}
    \label{fig1.0}
\end{subfigure}
\begin{subfigure}{0.5\textwidth}
    \def\svgwidth{0.9\columnwidth}
    \includegraphics[width=6.5cm]{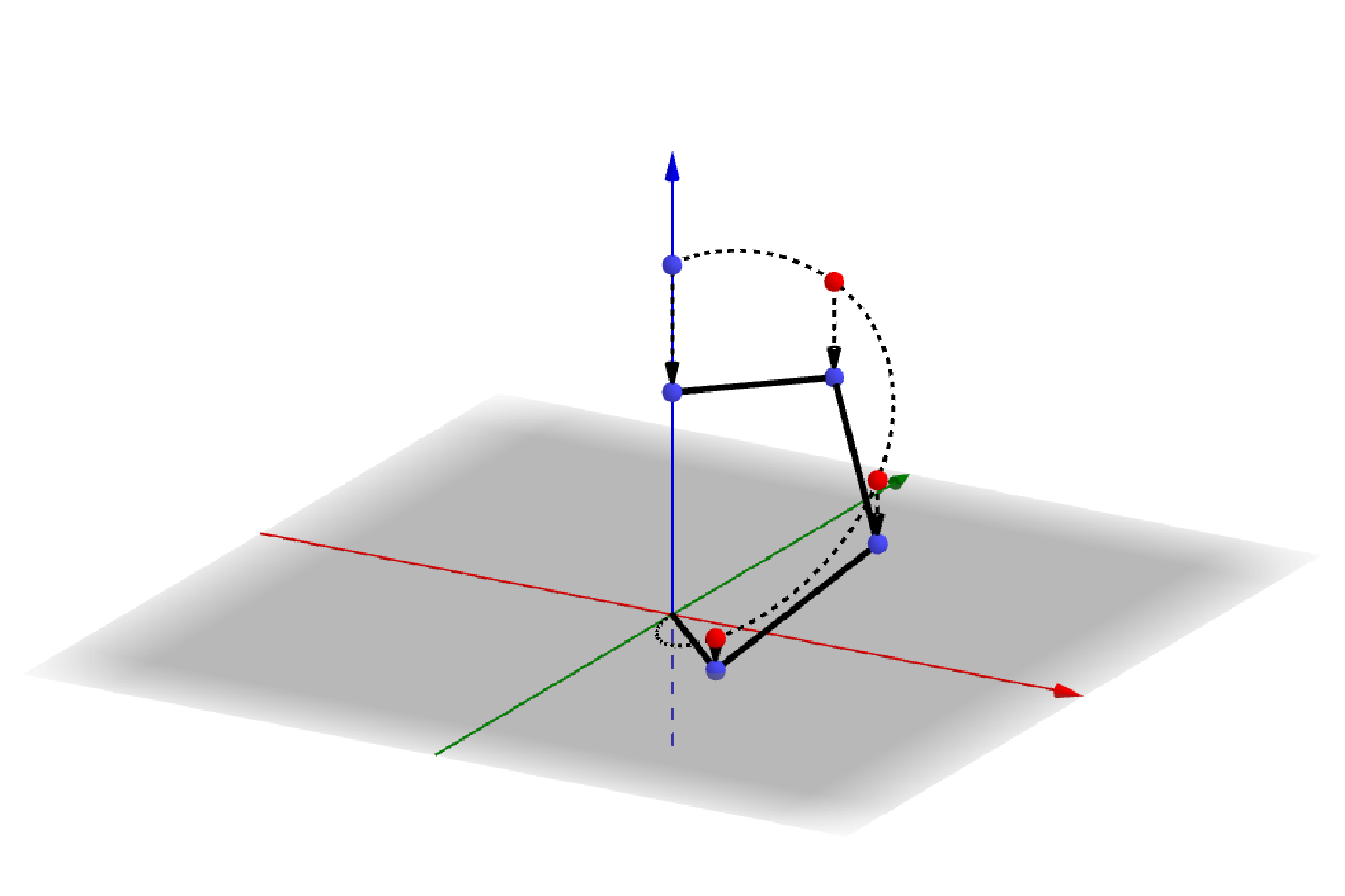}
    \caption{The path $\tilde{c}$}
    \label{fig1.1}
\end{subfigure}
\caption{}
\end{figure}

This $\tilde{c}$ is a horizontal path in $(N,d)$.
Let $\tilde{g}$ be the endpoint of $\tilde{c}$,
and $\tilde{h}_i$ the endpoint of $\tilde{c}_i$.
Hence $\tilde{h}_i=Y_{h_i}$ for $i\in I$ and $\tilde{h}_i=h_i$ for $i\notin I$.
By the triangle inequality,
$d\left(\tilde{g}\right)$ is bounded above by
\begin{equation}\label{equ1}
	\sum_{i\in I}d\left(\tilde{h}_i\right)+\sum_{i\notin I}\frac{t}{M}.
\end{equation}

By using (\ref{equ1}) and Lemma \ref{prop1},
\begin{align*}
d(g)-d\left(\tilde{g}\right)&\geq\sum_{i\in I}\left(\frac{t}{M}-d\left(\tilde{h}_i\right)\right)\\
		       &=\sum_{i\in I}\left(\frac{t}{M}-d\left(Y_{h_i}\right)\right)\\
&=\sum_{i\in I}\left(\frac{t}{M}-\|\pi(h_i)\|_{\infty}\right)\\
&\geq\frac{t}{M}\left(1-R\right)|I|.
\end{align*}
We shall see that $d(g)-d\left(\tilde{g}\right)$ is linearly bounded above by $t$.

Set $h=\tilde{g}^{-1}g$.
By the triangle inequality,
$$d(g)-d\left(\tilde{g}\right)\leq d(h).$$
Since each $\tilde{h}_i^{-1}h_i$ is in the center of $N$,
$$h=\tilde{g}^{-1}g=\tilde{h}_M^{-1}\cdots \tilde{h}_1^{-1}h_1\cdots h_M=\prod_{i\in I}\tilde{h}_i^{-1}h_i=\prod_{i\in I}Y_{h_i}^{-1}h_i\in[N,N].$$
By Lemma \ref{lem89},
we can choose $X,Y\in\partial (B_{\|\cdot\|}(1)))$ such that
\begin{itemize}
	\item $[\exp(X),\exp(Y)]=[X,Y]\in h^{\R_{>0}}\subset N$,
		and

	\item $\|[X,Y]\|_{[N,N]}\geq L$.
	\end{itemize}
Set $r\in\R_{\geq 0}$ such that
$$[\sqrt{r}X,\sqrt{r}Y]=r[X,Y]=h.$$
Then we can construct a horizontal path from $id$ to $h$ (equivalently,
can construct a path from $\tilde{g}$ to $g$ by translating the starting point) by connecting the following four paths:
$c_1(s)=-Xs$,
$c_2(s)=-Ys$,
$c_3(s)=Xs$ and $c_4(s)=Ys$ for $s\in[0,\sqrt{r}]$.
\begin{figure}[htbp]
	\centering
	\includegraphics[width=6.5cm]{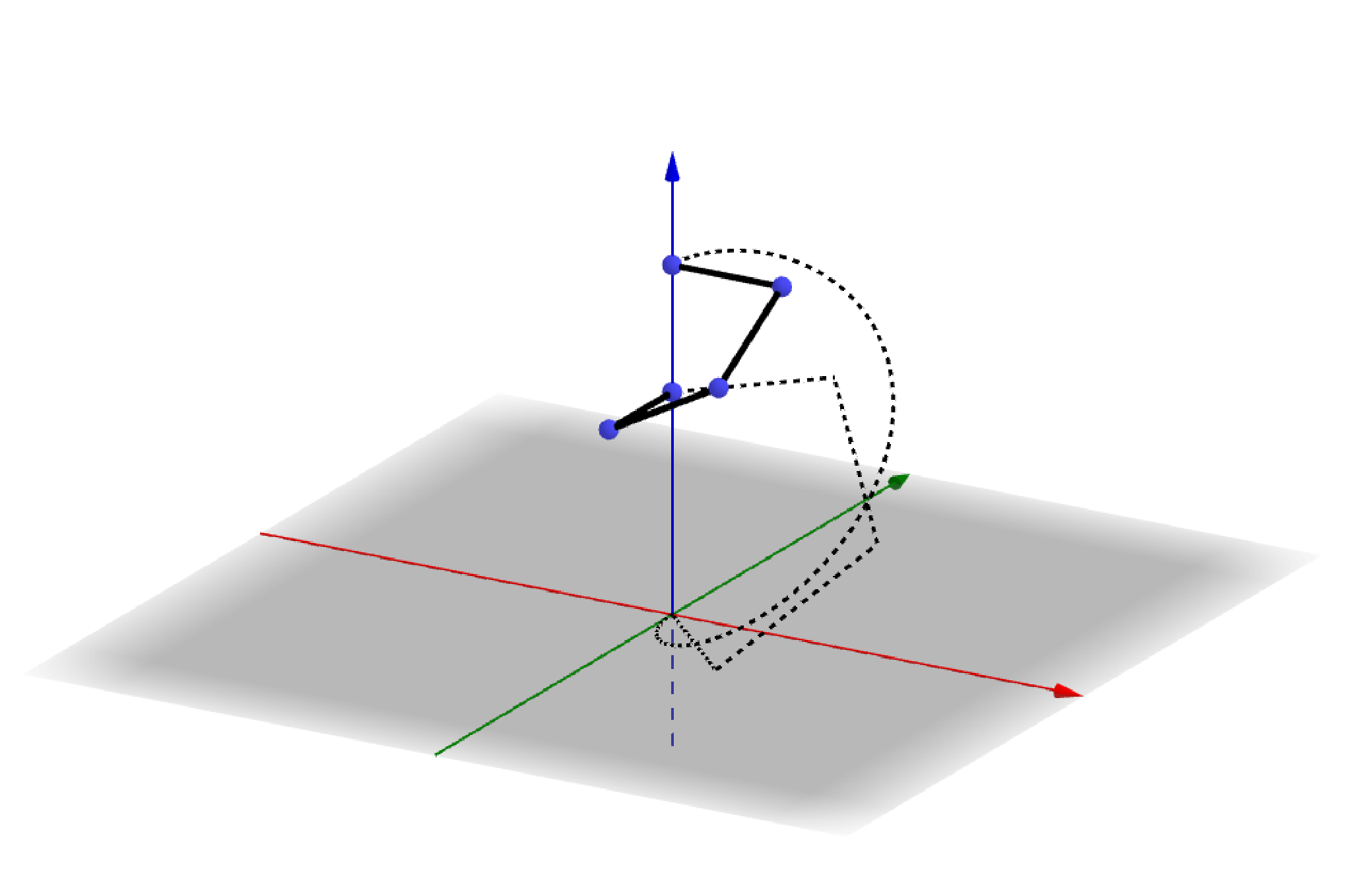}
	\caption{The path from $\tilde{g}$ to $g$}
\end{figure}

\noindent By the triangle inequality,
we obtain
$$d(h)\leq 4\sqrt{r}.$$

By the definition of $X,Y$ and $r$,
$\|h\|_{[N,N]}=\|r[X,Y]\|_{[N,N]}=r\|[X,Y]\|_{[N,N]}$.
Hence we obtain
$$r\leq \frac{\|h\|_{[N,N]}}{L}.$$
Finally we can estimate $\|h\|_{[N,N]}$ by using Lemma \ref{lem90} and Lemma \ref{lem91},
\begin{align*}
	\|h\|_{[N,N]}&=\|\prod_{i\in I}Y_{h_i}^{-1}h_i\|_{[N,N]}\\
	     &\leq\sum_{i\in I}\|Y_{h_i}^{-1}h_i\|_{[N,N]}\\
	     &\leq |I|K_1\left(\max\left\{d_{\infty}(h_i),d_{\infty}(Y_{h_i})\right\}\right)^2\\
	     &\leq |I|K_1\left(\max\left\{K_2d(h_i)+K_2,K_2d(Y_{h_i})+K_2\right\}\right)^2\\
	     &\leq 4K_1K_2^2\frac{t^2}{M^2}|I|.
\end{align*}

To be summarized,
\begin{align*}
\frac{t}{M}(1-R)|I|&\leq d(g)-d\left(\tilde{g}\right)\\
&\leq d(h)\\
&\leq 4\sqrt{r}\\
&\leq 4\sqrt{\frac{\|h\|_{[N,N]}}{L}}\\
&\leq 8K_2\frac{t}{M}\sqrt{\frac{K_1|I|}{L}}.
\end{align*}

Solve the quadratic inequality for $\sqrt{|I|}$,
then we have
$$|I|\leq\frac{64K_1K_2^2}{L(1-R)^2}=\frac{K}{(1-R)^2},$$

where $K=\frac{64K_1K_2^2}{L}$.

\end{proof}
\begin{rmk}
Another choice of a norm may inherit another constant $K>0$,
however it does not affect the later arguments.
If necessary,
we can take the infimum one among obtained $K$ since our method can be applied to any norm.
\end{rmk}

\section{Proof of the main theorem}

In the arguments of Section \ref{sec9},
one does not need the non-singularity.
If $N$ is non-singular,
then we obtain the following lemma.

\begin{lem}\label{lem92}

	There exists $L_0>0$ such that for all $r_1,r_2\in\R_{>0}$ and all $g\in \pi^{-1}(\partial B_{\|\cdot\|_{\infty}}(0,r_1))$,
$$B_{\|\cdot\|_{[N,N]}}(L_0r_1r_2)\subset[g,B_{d_{\infty}}(r_2)].$$

\end{lem}

Before the proof,
we confirm some easy facts on $[g,B_{d_{\infty}}(r_2)]$.

\begin{sub}\label{sub91}
	$[g,B_{d_{\infty}}(r_2)]$ is a compact star convex neighborhood around $id\in[N,N]$.
\end{sub}
\begin{proof}

$[g,B_{d_{\infty}}(r_2)]$ is compact since the mapping $[g,\cdot]:N\to[N,N]$ is continuous and $B_{d_{\infty}}(r_2)$ is compact.
Moreover it is a neighborhood around the identity since $[g,\cdot]$ is a submersion by non-singular condition.

Next we show the set is star convex.
For any $h\in[g,B_{d_{\infty}}(r_2)]$,
we can choose $Y\in B_{\|\cdot\|_{\infty}}(r_2)$ such that $h=[g,\exp(Y)]$.
By Campbell-Baker-Haudorff formula,
for $s\in[0,1]$,
$$h^s=[g,\exp(Y)]^s=[g,\exp(sY)].$$
By Fact \ref{fact20}(b),
$$d_{\infty}(\exp(sY))=sd_{\infty}(\exp(Y))\leq r_2.$$
Hence $h^s\in[g,B_{d_{\infty}}(r_2)]$,
that is,
$[g,B_{d_{\infty}}(r_2)]$ is star convex.

\end{proof}

\begin{rmk}\label{rmk90}
	From the proof of Sublemma \ref{sub91},
	we have
	$$\partial[g,B_{d_{\infty}}(r_2)]=[g,\exp(\partial B_{\|\cdot\|_{\infty}}(r_2))].$$
\end{rmk}

\begin{proof}[Proof of Lemma \ref{lem92}]
First of all,
we find $L_0>0$ such that for any\\
$g\in \pi^{-1}\left(\partial B_{\|\cdot\|_{\infty}}(1)\right)$,
$B_{\|\cdot\|_{[N,N]}}(L_0)\subset[g,B_{d_{\infty}}(1)]$.

By Sublemma \ref{sub91},
there exists $L(g)>0$ such that
$$B_{\|\cdot\|_{[N,N]}}(L(g))\subset[g,B_{d_{\infty}}(1)].$$
We can assume $L(g_1)=L(g_2)$ if $\pi(g_1)=\pi(g_2)$,
since\\
$[g_1,B_{d_{\infty}}(1)]=[g_2,B_{d_{\infty}}(1)]$.
Since $[\cdot,\cdot]$ is continuous,
we may take $L(g)$ continuously for $g\in\pi^{-1}(\partial B_{\|\cdot\|_{\infty}}(1))$.
Hence
$$L_0=\min\{L(g)|g\in\pi^{-1}(\partial B_{\|\cdot\|_{\infty}}(1))\}$$
exists and is non-zero.

Next we consider the general case.
Since $[g,B_{d_{\infty}}(r_2)]$ is star convex,
we only need to show that all points at the boundary of $[g,B_{d_{\infty}}(r_2)]$ are at least $L_0r_1r_2$ away from the identity.\\
$h\in\partial[g,B_{d_{\infty}}(r_2)]$ is represented by $h=[\exp(\pi(g)),\exp(Y)]$,
where $Y\in\partial B_{\|\cdot\|_{\infty}}(r_2)$ as we mentioned in Remark \ref{rmk90}.
Set $X'=\frac{1}{r_1}\pi(g)$ and $Y'=\frac{1}{r_2}Y$,
then we have
$$\|h\|_{[N,N]}=\|[\exp(\pi(g)),\exp(Y)]\|_{[N,N]}=r_1r_2\|[\exp(X'),\exp(Y')]\|_{[N,N]}\geq L_0r_1r_2.$$

\end{proof}

\begin{rmk}\label{rmkinv}

We can replace $[g,B_{d_{\infty}}(r_2)]$ to $[g,B_d(r_2)]$ in Lemma \ref{lem92},
since $\pi(B_d(r_2))=\pi(B_{d_{\infty}}(r_2))$ implies
$$[g,B_{d_{\infty}}(r_2)]=[g,B_d(r_2)].$$

\end{rmk}

By using the previous lemmas,
we show the following theorem,
which is a precise statement of Theorem \ref{thm3}.

\begin{thm}\label{thmneo}
	Let $N$ be a simply connected non-singular $2$-step nilpotent Lie group endowed with a left invariant subFinsler metric $d$,
	and $(N,d_{\infty},id)$ the asymptotic cone of $(N,d,id)$.
	Then there is $C>0$ such that for any $g\in N$,
	$$\left|d(g)-d_{\infty}(g)\right|<C.$$
\end{thm}

\begin{proof}

First we show that $d_{\infty}(g)-d(g)$ is uniformly bounded above.
Fix $0<R<1$ and $M>0$ sufficiently large so that $M-|I(c,M,R)|\neq0$ for any geodesic $c$ with $length(c)\geq M$.
It is possible by Proposition \ref{prop90}.
It suffices to show the case where $g\in N\setminus B_d(M)$,
since $d$ and $d_{\infty}$ are proper metrics on $N$.

Let $t=d(g)$ and $c$ a geodesic from $id$ to $g$ in $(N,d)$.
We will construct a horizontal path $\breve{c}$ in $(N,d_{\infty})$ which starts at the identity and ends at $g$,
and show that the length of $\breve{c}$ is not so long relative to that of $c$.

It needs two steps to construct a path $\breve{c}$.
First,
Deform $c$ into a horizontal path $\tilde{c}$ in $(N,d_{\infty})$ as follows.

\begin{itemize}

\item[(1)] Divide $c$ into $M$ pieces of geodesics $c_i$ as in Proposition \ref{prop90}.
Since $d$ is left invariant,
we may see each $c_i$ a geodesic from $id$ to $h_i\in N$ with $h_1\cdots h_M=g$.
\item[(2)] Divide $c_i$ into $m=\left[\frac{t}{M}\right]$ pieces of geodesics $c_{ij}$ and set $h_{ij}$ in the same way,
	where $[~\cdot~]$ is the Gaussian symbol.
\item[(3)] Set $\tilde{c}$ the concatenation of $\tilde{c_{ij}}(s)=s\frac{\pi(h_{ij})}{\|\pi(h_{ij})\|_{\infty}}$,
	$s\in[0,\|\pi(h_{ij})\|_{\infty}]$.
	By Lemma \ref{prop1},
	$length\left(\tilde{c}\right)=\sum\|\pi(h_{ij})\|_{\infty}\leq\sum d(h_{ij})=d(g)$.
\item[(4)] Let $\tilde{c}_i$ be the concatenation of paths $\tilde{c}_{i1},\dots,\tilde{c}_{im}$.
\end{itemize}

Let $\tilde{g}$ be the endpoint of $\tilde{c}$ and set $h=\tilde{g}^{-1}g$.
Since $\pi(h_{ij})^{-1}h_{ij}\in[N,N]$,
$$h=\pi(h_{Mm})^{-1}\cdots \pi(h_{11})^{-1}h_{11}\cdots h_{Mm}\in[N,N].$$

By using the path $\tilde{c}$,
we shall construct a horizontal path $\breve{c}$.
By definition of $I$,
$\|\pi(h_i)\|_{\infty}\geq Rd(h_i)=R\frac{t}{M}$ for $i\notin I$.
In particular,
$h_i\notin[N,N]$ for $i\notin I$.
By Lemma \ref{lem92},
there exists $X_i\in\partial\pi(B_{\|\cdot\|_{\infty}}(1)))$ such that $[h_i,X_i]\in h^{\R_{>0}}$ and that $\|[h_i,X_i]\|_{[N,N]}\geq L_0$.
Set $r\in\R_{\geq 0}$ so that
$$\prod_{i\notin I}[h_i,rX_i]=(\prod_{i\notin I}[h_i,X_i])^r=h.$$
Define $\breve{c}$ as the concatenation of $\breve{c}_i$,
$i=1,\dots,M$,
given as follows.
\begin{itemize}
	\item[(1)] For $i\notin I$,
		let $\breve{c}_i$ be a concatenation of three paths;
		$\breve{c}_{i1}(s)=-sX_i ~~(s\in[0,r])$,
		$\tilde{c}_i$,
		and $\breve{c}_{i2}(s)=sX_i ~~(s\in[0,r])$.
		Hence the length of $\breve{c}_i$ is $2r+\sum_j\|\pi(h_{ij})\|_{\infty}$.
	\item[(2)] For $i\in I$,
		set $\breve{c}_i=\tilde{c}_i$.
\end{itemize}

\begin{figure}[h]
\begin{subfigure}{0.5\textwidth}
    \def\svgwidth{0.9\columnwidth}
    \includegraphics[width=6.5cm]{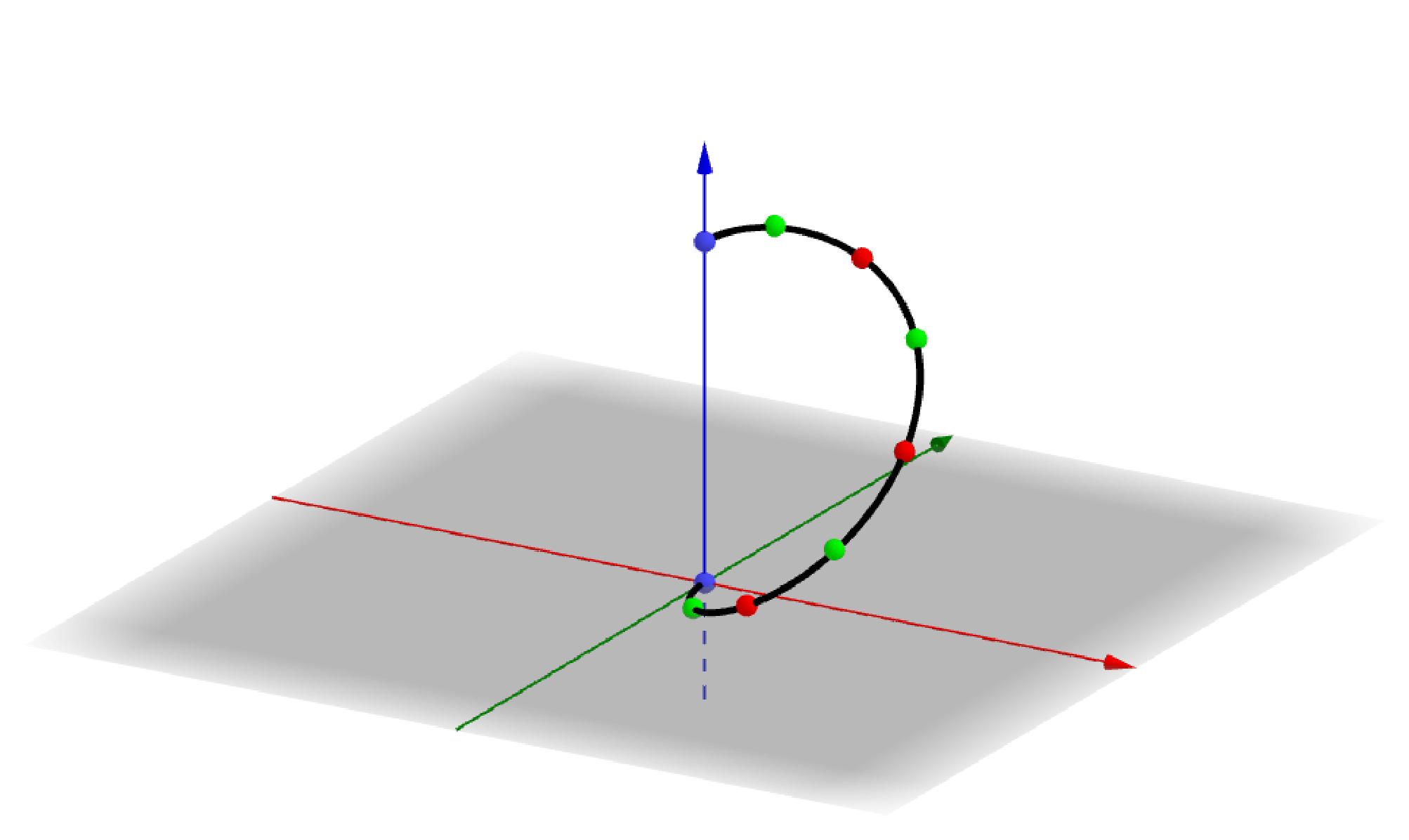}
    \caption{The path $c$}
    \label{fig2.1}
\end{subfigure}
\begin{subfigure}{0.5\textwidth}
    \def\svgwidth{0.9\columnwidth}
    \includegraphics[width=6.5cm]{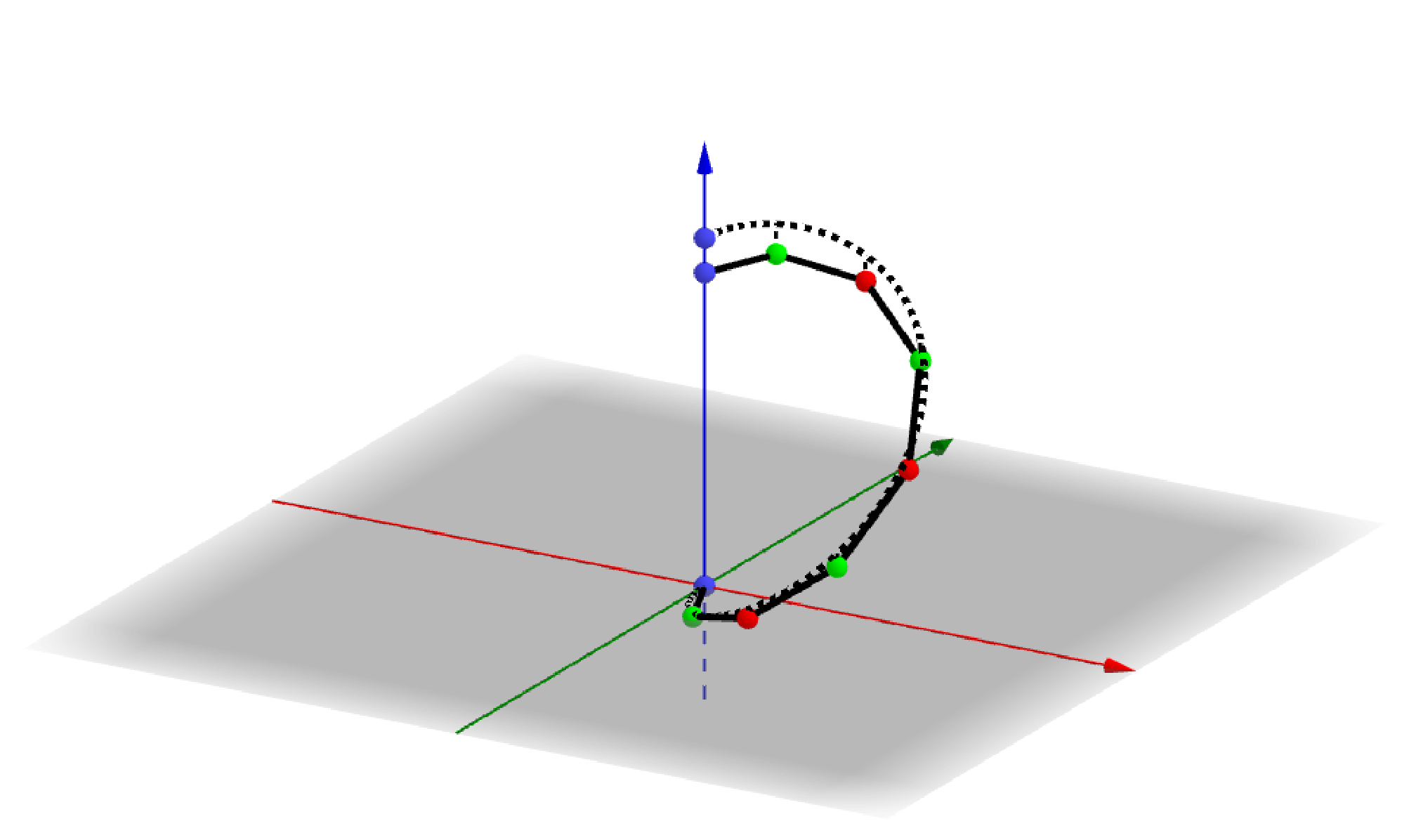}
    \caption{The path $\tilde{c}$}
    \label{fig2.2}
\end{subfigure}
\caption{}
\end{figure}

\begin{figure}[h]
\begin{subfigure}{0.5\textwidth}
    \def\svgwidth{0.9\columnwidth}
    \includegraphics[width=6.5cm]{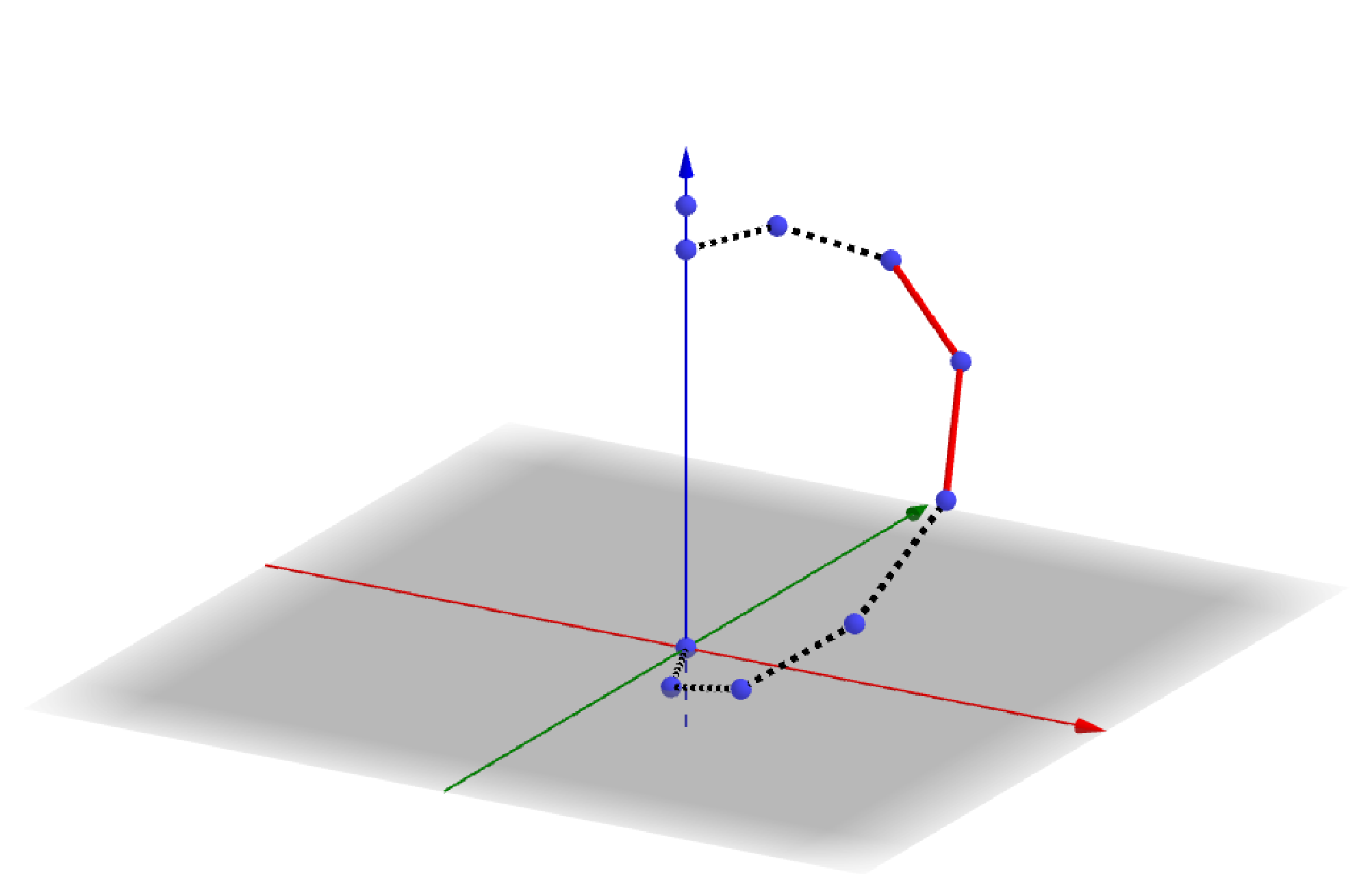}
    \caption{The subpath $\tilde{c}_i$ for $i\notin I$}
    \label{fig2.3}
\end{subfigure}
\begin{subfigure}{0.5\textwidth}
    \def\svgwidth{0.9\columnwidth}
    \includegraphics[width=6.5cm]{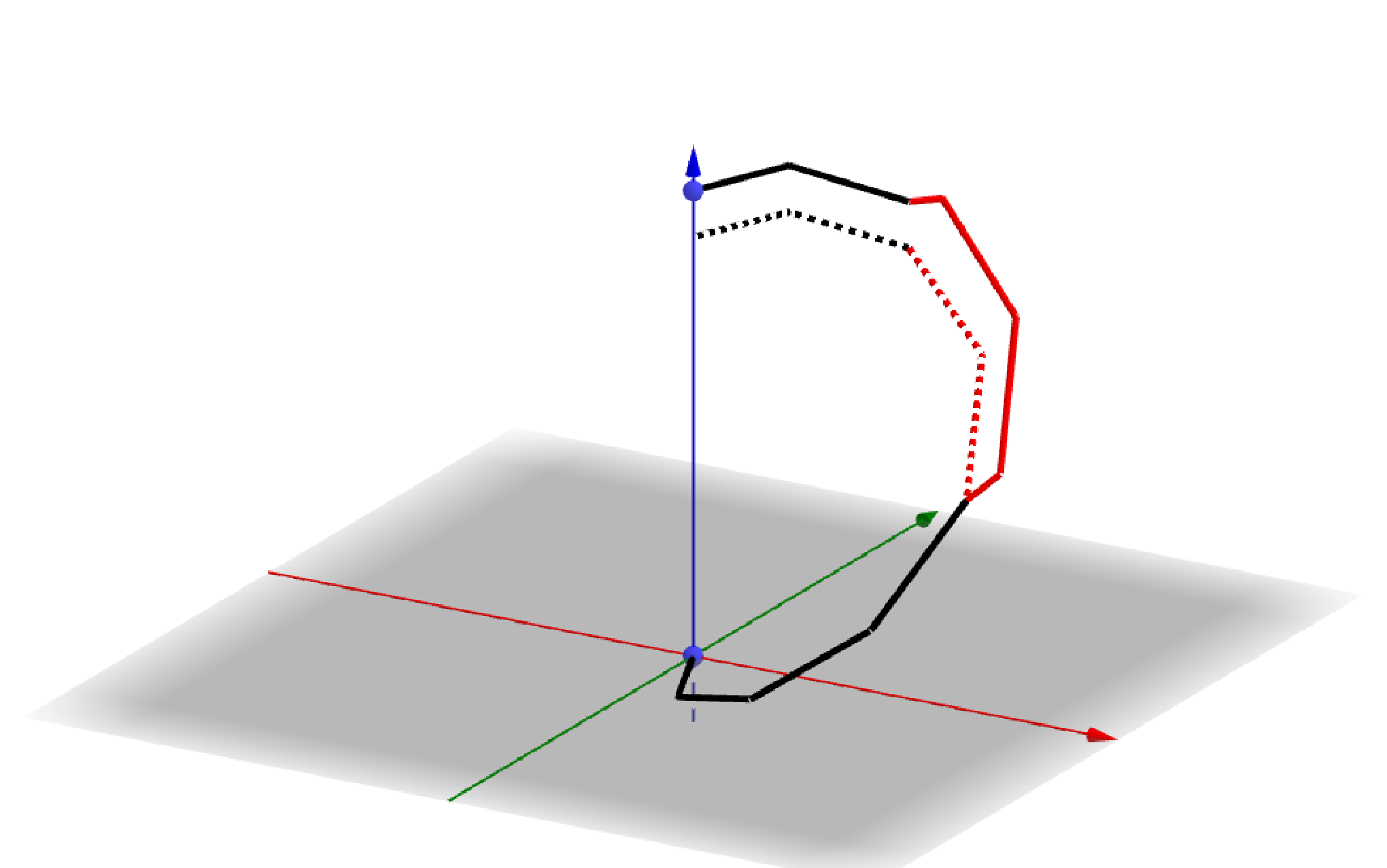}
    \caption{The path $\breve{c}$}
    \label{fig2.4}
\end{subfigure}
\caption{}
\end{figure}

This path $\breve{c}$ starts at the identity and ends at $g$ by the Campbell-Baker-Hausdorff formula.
The length of $\breve{c}$ is
$$length\left(\breve{c}\right)=\sum\|\pi(h_{ij})\|_{\infty}+2r(M-|I|)\leq d(g)+2r(M-|I|).$$
The rest of the proof is to find an upper bound of $r(M-|I|)$.

By Lemma \ref{lem92},
	$$\|[h_i,X_i]\|_{[N,N]}=\|[\pi(h_i),X_i]\|_{[N,N]}\geq R\frac{t}{M}L_0.$$
Since each $[h_i,rX_i]$ is in $h^{\R_{>0}}$,
we obtain
\begin{equation}\label{eq110}
	\|h\|_{[N,N]}=\left\|\prod_{i\notin I}[h_i,rX_i]\right\|_{N,N]}=r\sum_{i\notin I}\|[h_i,X_i]\|_{[N,N]}\geq r(M-|I|)R\frac{t}{M}L_0.
\end{equation}
Hence our goal is changed to find an upper bound of $\|h\|_{[N,N]}$.
By using Lemma \ref{lem90} and Lemma \ref{lem91},

\begin{align*}
	\|h\|_{[N,N]}&=\left\|\pi(h_{mM})^{-1}\cdots \pi(h_{11})^{-1}h_{11}\cdots h_{mM}\right\|_{[N,N]}\\
		     &=\left\|\prod\pi(h_{ij})^{-1}h_{ij}\right\|_{[N,N]}\\
&\leq\sum\|\pi(h_{ij})^{-1}h_{ij}\|_{[N,N]}\\
&\leq mM\sup\left\{\|g^{-1}h\|_{[N,N]}~\bigg\vert~g,h\in B_d\left(\frac{t}{mM}\right),g^{-1}h\in[N,N]\right\}\\
&\leq mM\sup\left\{\|g^{-1}h\|_{[N,N]}~\bigg\vert~g,h\in B_{d_{\infty}}\left(K_2\frac{t}{mM}+K_2\right),g^{-1}h\in[N,N]\right\}\\
&\leq mMK_1\left(K_2\frac{t}{mM}+K_2\right)^2\\
&\leq 4mMK_1K_2^2\frac{t^2}{m^2M^2}\\
&\leq 4tK_1K_2^2.
\end{align*}
Hence $\|h\|_{[N,N]}$ is linearly bounded by $t$.

Combining with the equation(\ref{eq110}),
we obtain
$$r(M-|I|)\leq \frac{4K_1K_2^2M}{RL_0}.$$
We have constructed a path $\breve{c}$ which is sufficiently short relative to the original path $c$,
hence we have
$$d_{\infty}(g)\leq d(g)+\frac{8K_1K_2^2M}{RL_0}.$$

The other side of the inequality follows in a similar way.
The difference is only the construction of $\tilde{c}$ and $\breve{c}$.
In the construction of $\tilde{c}$,
we let $\tilde{c}_{ij}(s)=s\frac{Y_{h_{ij}}}{\|Y_{h_{ij}}\|}$.
In the construction of $\breve{c}$,
we let $\breve{c}_{i1}(s)=-sY_{X_i}$ and $\breve{c}_{i2}(s)=sY_{X_i}$.
The rest of the proof follows in the same way.
\end{proof}


\begin{thebibliography}{99}

\bibitem{ban} Y. Bungert, Minimal geodesics, \textit{Ergod. Th. Dynam. Sys}. \textbf{10}, 263-286, 1989.

\bibitem{bre} E. Breuillard, Geometry of locally compact groups of polynomial growth and shape of large balls, {\it Groups, Geom. Dynam.}, {\bf 8} (3), 669-732, 2014.

\bibitem{bre2} E. Breuillard and E. Le Donne, Nilpotent groups, asymptotic cone and subFinsler geometry, https://www.math.u-psud.fr/~breuilla/Balls.pdf, 2012.

\bibitem{bur} D. Burago, Periodic metrics, {\it Representation Theory and Dynamical Systems, Adv. Soviet Math.}, {\bf 9}, 205-210, 1992.

\bibitem {bur2} D. Burago, Periodic metrics, in \textit{Seminar on Dynamical Systems}, Progress in nonlinear differential equations (H.Brezis, ed.), vol.\textbf{12}, \textit{Birkh\"auser}, 90–95, 1994.

\bibitem{duc} M. Duchin and C. Mooney, Fine asymptotic geometry in the Heisenberg group, {\it Indiana univ. Math jounal,} {\bf 63}(3), 885-916, 2014.

\bibitem{fuj2} K. Fujiwara, Can one hear the shape of a group? {\it Geom. and Topol. of Manifolds}, 10th China-Japan Conference 2014, 139-146, 2016.

\bibitem{gia} V. Gianella, On the asymptotics of the growth of nilpotent groups, {\it Doctorial thesis in research-collection.ethz.ch}, 2017.

\bibitem{gro2} M. Gromov, Asymptotic invariants of infinite groups, in {\it Geometric group theory,} Vol.2 (Sussex 1991), London Math. Soc. Lecture Note Ser. {\bf 182}, 1-295, 1993.

\bibitem{kra} S.A. Krat, Asymptotic properties of the Heisenberg group, {\it J. Math. Sci.}, {\bf 110} (4), 2824-2840, 2002.

\bibitem{led} E. Le Donne and S. Nicolussi Golo, Regularity properties of spheres in homogeneous groups, \textit{Trans. Am. Math. Soc.} \textbf{370}(3), 2057–2084, 2018.
\bibitem{pan} P. Pansu, Croissance des boules et des g\'eod\'esiques ferm\'ees dans les nilvari\'et\'es, {\it Ergodic Theory Dynam. Systems,} {\bf 3}(3), 415-445, 1993.

\bibitem{sto2} M. Stoll, Coarse length can be unbounded in $3$-step nilpotent Lie groups, preprint, 2010.

\bibitem{sto} M. Stoll, On the asymptotics of the growth of $2$-step nilpotent groups, {\it London Math. Soc.}, {\bf 58} (1), 38-48, 1998.

\bibitem{van} L. van der dries and A. J. Wilkie, Gromov's theorem on groups of polynomial growth and elementary logic, {\it J. Algebra}, {\bf 89} (2), 349-374, 1984.

\end{thebibliography}
\end{document}